\documentclass[12pt]{amsart}

\usepackage{fullpage}
\usepackage{amsfonts}
\usepackage{verbatim}
\newcommand{\swoosh}{\includegraphics[width=0.35in]{swoosh.pdf}}

\newcommand{\R}{\mathbb{R}}

\newcommand{\Z}{\mathbb{Z}}

\newcommand{\T}{\mathbb{T}}

\newcommand{\Hil}{\mathcal{H}}
\newcommand{\Kil}{\mathcal{K}}

\providecommand{\norm}[1]{\lVert#1\rVert}

\numberwithin{equation}{section}
\theoremstyle{plain} %% This is the default, anyway

\newtheorem{thm}{Theorem}[section]
\newtheorem{cor}[thm]{Corollary}

\newtheorem{prop}[thm]{Proposition}

\theoremstyle{definition}
\newtheorem{defn}{Definition}[section]

\theoremstyle{remark}
\newtheorem{remark}{Remark}[section]

\begin{document}

\title{Orthonormal dilations of non-tight frames}

\author{Marcin Bownik}
\author{John Jasper}

\address{Department of Mathematics, University of Oregon, Eugene, OR 97403--1222, USA}

\email{mbownik@uoregon.edu, jjasper@uoregon.edu}

\begin{abstract}
We establish dilation theorems for non-tight frames with additional structure, i.e., frames generated by unitary groups of operators and projective unitary representations. This generalizes previous dilation results for Parseval frames due to Han and Larson \cite{hl} and Gabardo and Han \cite{GH}. We also extend the dilation theorem for Parseval wavelets, due to Dutkay, Han, Picioroaga, and Sun \cite{dhps}, by identifying the optimal class of frame wavelets for which dilation into an orthonormal wavelet is possible.
\end{abstract}

\author{Darrin Speegle}

\address{Department of Mathematics and Computer Science, Saint Louis 
University, 221 N. Grand Blvd.,
St. Louis, MO 63103, USA}

\email{speegled@slu.edu}


\thanks{The first two authors
were partially supported by NSF grant DMS-0653881.}

\date{\today}

\maketitle

\section{Introduction}

\begin{defn} A sequence $\{f_{i}\}_{i\in I}$ in a Hilbert space $\Hil$ is called a {\it frame} if there exist $0<A\leq B<\infty$ such that
\begin{equation}\label{frame}A\norm{f}^{2}\leq\sum |\langle f,f_{i}\rangle|^{2}\leq B\norm{f}^{2}
\qquad\text{for all $f\in\Hil$}.
\end{equation}
The numbers $A$ and $B$ are called the {\it frame bounds}. The supremum over all $A$'s and infimum over all $B$'s which satisfy \eqref{frame} are called the {\it optimal frame bounds}. If $A=B$, then $\{f_i\}$ is said to be a {\it tight frame}. In addition, if $A=B=1$, then $\{f_i\}$ is called a {\it Parseval frame}.
\end{defn}

The most basic dilation result for frames says that a Parseval frame in a Hilbert space $\mathcal H$ is an image of an orthonormal basis under an orthogonal projection of some larger Hilbert space $\mathcal K \supset \mathcal H$ onto $\mathcal H$. This is now a classical fact which can be attributed to Han and Larson \cite{hl}, who also proved the following result. If the Parseval frame has some additional structure, i.e., it is generated by an action of a unitary group of operators on $\mathcal H$, then the corresponding orthonormal basis is also generated by a unitary group of operators on $\mathcal K$. Gabardo and Han \cite{GH} proved that similar results hold for group-like unitary operator systems, which include Gabor systems. Even more general frames generated by projective unitary representations were studied by Han and Larson \cite{HL08}. Another remarkable result in this direction is due to Dutkay, Han, Picioroaga, and Sun \cite{dhps} who established a dilation theorem for Parseval wavelets.

Despite this progress, much less was known about dilation results for non-tight frames. Han and Larson \cite[Proposition 1.6]{hl} proved that any frame is an image of a Riesz basis under an orthogonal projection, and that the frame and Riesz bounds are the same. Recently, Bownik and Jasper \cite[Proposition 2.3]{bj} proved a dilation result for non-tight frames which is also implicitly contained in the work of Antezana, Massey, Ruiz, and Stojanoff \cite[Proposition 4.5]{amrs}. This result says that any frame with frame bounds $A$ and $B$ is the image of an orthonormal basis under a positive (self-adjoint) operator with spectrum contained in $\{0\} \cup [\sqrt{A}, \sqrt{B}]$. In the case of Parseval frames, when $A=B=1$, this easily reduces to the classical Han and Larson dilation theorem.

The goal of this paper is to extend the dilation theorem \cite{bj} to non-tight frames with some additional structure. We mainly concentrate on frames and orthonormal bases generated either by an action of a unitary group or by a projective unitary representation. We show that any such frame is the image of an orthonormal basis under a positive operator with suitable spectrum.  Moreover, it is possible to choose the orthonormal basis so as to have the same structure as the frame.  In the case of Parseval frames, the positive operator is actually a projection; hence, our results generalize those of Han, Gabardo, and Larson. Finally, we also extend the dilation  theorem for Parseval wavelets in \cite{dhps}, albeit with an additional twist. It turns out that orthonormal dilation is possible only for a proper subset of frame wavelets, which contains all Parseval wavelets. We characterize the class of frame wavelets for which a generalization of \cite[Theorem 2.6]{dhps} holds.

\section{Frames vectors for unitary groups}

In this section we establish the dilation theorem for frame vectors generated by a unitary group. Theorem \ref{ug} generalizes the result of Han and Larson \cite[Theorem 3.8]{hl} from the setting of Parseval frames to general non-tight frames.

\begin{defn}
Let $\mathcal{U}$ be a set of unitary operators on a Hilbert space $\Hil$. We say that $f\in\Hil$ is a {\it frame vector} for $\mathcal U$  if $\{Uf\}_{U\in\mathcal U}$ is a frame for $\Hil$. We say that $e\in \Hil$ is a complete wandering vector if  $\{Ue\}_{U\in\mathcal U}$ is an orthonormal basis of $\Hil$.
\end{defn}

\begin{thm}\label{ug} Let $\mathcal{U}$ be a group of unitary operators on a Hilbert space $\Hil$. Suppose that $f\in\Hil$ is a frame vector for $\mathcal U$ with optimal bounds $A^{2}$ and $B^{2}$.
Then, there exist:
\begin{enumerate}
\item 
a Hilbert space $\mathcal{K}\supseteq\Hil$ and a unitary group $\mathcal{V}$ on $\mathcal{K}$, such that the restriction map $\mathcal V \ni V \mapsto V|_{\mathcal H}$ is a group isomorphism of $\mathcal V$ onto $\mathcal U$,
\item
a positive operator $E:\mathcal{K}\to\Hil$ such that $\{A,B\}\subseteq\sigma(E|_{\mathcal H})\subseteq [A,B]$, and
\item
a complete wandering vector $e\in\mathcal K$ for $\mathcal V$ such that $E\mathcal Ve=\mathcal Vf =\mathcal Uf$.
\end{enumerate}
More precisely, we have
$
E(Ve) =Vf$ for all $V \in\mathcal V$.
\end{thm}

We will use the following standard terminology.

\begin{defn} If $\{f_{i}\}_{i\in I}$ is a frame we call the operator $T:\Hil\to\ell^{2}(I)$, given by
\begin{equation}\label{analysis}
Tf = \{\langle f,f_{i}\rangle\}_{i\in I}\end{equation}
the {\it analysis operator}. The adjoint $T^{\ast}:\ell^{2}(I)\to\Hil$ given by
\begin{equation}\label{synthesis}
T^{\ast}\big(\{a_{i}\}_{i\in I}\big)=\sum_{i\in I}a_{i}f_{i}\end{equation}
is called the {\it synthesis operator}. The operator $S=T^{\ast}T$ given by
\begin{equation}\label{frameop}
Sf = \sum_{i\in I}\langle f,f_{i}\rangle f_{i}\end{equation}
is called the {\it frame operator}.\end{defn}

The following is a standard fact about frames.

\begin{prop}\label{di1} If $\{f_{i}\}_{i\in I}$ is a frame for $\Hil$ with frame operator $S$, then $\{S^{-1/2}f_{i}\}_{i\in I}$ is a Parseval frame for $\Hil$. In this case, 
$\{S^{-1/2}f_{i}\}_{i\in I}$ is said to be the canonical Parseval frame of  $\{f_{i}\}_{i\in I}$.
\end{prop}

For the purpose of the proof it is convenient to reformulate Theorem \ref{ug} in more explicit terms as follows. 
Let $\mathcal{U}$ be a group of unitary operators on a Hilbert space $\Hil$. Let $\{e_{U}\}_{U\in\mathcal{U}}$ be the coordinate basis of $\ell^2(\mathcal{U})$. For each $U\in\mathcal{U}$, we define the unitary operator $\lambda_U$ on $\ell^2(\mathcal{U})$ by $\lambda_{U}e_V=e_{UV}$. The map $U\mapsto\lambda_{U}$ is often called the left-regular unitary representation of $\mathcal U$. Define the unitary group $\mathcal{V}=\{\lambda_U:U\in\mathcal{U}\}$. Finally we are ready to state and prove the explicit version of Theorem \ref{ug}.

\begin{thm}\label{ug2} 
Suppose that $f\in\Hil$ is frame vector for $\mathcal U$ with optimal bounds $A^{2}$ and $B^{2}$.
Then,
\begin{enumerate}
\item 
there exists an isometry $\Phi:\Hil\to\ell^{2}(\mathcal{U})$ such that $\Phi^* \lambda_U \Phi = U$ for all $U \in\mathcal U$,
\item
there exists a positive operator $E:\ell^{2}(\mathcal{U})\to\Phi(\Hil)$ such that $\{A,B\}\subseteq\sigma(E|_{\Phi(\Hil)})\subseteq [A,B]$, and
\item
we have $E(\mathcal V e_I)=\Phi(\mathcal Uf)$, where $I$ is the identity on $\mathcal H$. More precisely, 
\[
E e_{U} =E \lambda_U e_I =\Phi Uf \qquad\text{for all }U \in\mathcal U.
\]
\end{enumerate}
Moreover, $E^2$ is the Gramian of $\mathcal U f$, i.e., $\langle E^2 e_U, e_V \rangle = \langle Uf, Vf\rangle$ for all $U,V \in \mathcal U$. In particular, $E^{2}$ is unitarily equivalent to $S\oplus\mathbf{0}$, where $S$ is the frame operator of $\mathcal Uf$ and $\mathbf{0}$ is the zero operator on $\Phi(\Hil)^{\bot}$.
\end{thm}

Note that Theorem \ref{ug2} immediately implies Theorem \ref{ug} by identifying $\Hil$ with $\Phi(\Hil)$, by letting $\mathcal K=\ell^2(\mathcal U)$, $\mathcal V = \{\lambda_U: U \in \mathcal U\}$, and by setting $e=e_I$, which is a complete wandering vector for $\mathcal{V}$. 

\begin{proof} Let $S$ be the frame operator of $\mathcal{U}f$. By Proposition \ref{di1}, $S^{-1/2}\mathcal{U}f$ is a Parseval frame. For $U\in\mathcal{U}$ set $p_{U}=S^{-1/2}Uf$, and let $\Phi$ be the analysis operator of $\{p_{U}\}_{U\in\mathcal{U}}$, that is
\[
\Phi g = \sum_{U\in\mathcal{U}}\langle g,p_{U}\rangle e_{U}\qquad\text{for all }g\in\Hil.
\]
Since $\{p_{U}\}_{U\in\mathcal{U}}$ is a Parseval frame, $\Phi$ is an isometry.
First, we show that $SU=US$ for each $U\in\mathcal{U}$. Indeed, for $g\in\Hil$
\begin{align*}
USg & =U\left(\sum_{V\in\mathcal{U}}\langle g,Vf\rangle Vf\right) = \sum_{V\in\mathcal{U}}\langle g,Vf\rangle UVf = \sum_{V\in\mathcal{U}}\langle Ug,UVf\rangle UVf  = SUg.
\end{align*}
Since $S$ is self-adjoint, we also have $S^{-1/2}U=US^{-1/2}$ for each $U\in\mathcal{U}$. Thus, for $g\in\Hil$
\begin{align*}
\lambda_U \Phi g & = \lambda_U\left(\sum_{V\in\mathcal{U}}\langle g,S^{-1/2}Vf\rangle e_V\right)
 = \sum_{V\in\mathcal{U}}\langle Ug,US^{-1/2}Vf\rangle e_{UV}\\
 & = \sum_{V\in\mathcal{U}}\langle Ug,S^{-1/2}UVf\rangle e_{UV} = \sum_{V\in\mathcal{U}}\langle Ug,p_{UV}\rangle e_{UV} = \Phi Ug.
\end{align*}
Since $\Phi^*\Phi$ is the identity on $\Hil$, this shows (i).

Next, note that $\Phi^{\ast}$ is the synthesis operator of $\{p_{U}\}_{U \in\mathcal U}$ given by
\[
\Phi^{\ast}g = \sum_{U\in\mathcal{U}}\langle g,e_{U}\rangle p_{U}\qquad\text{for all }g\in\ell^{2}(\mathcal{U}).
\]
Hence, we have $\Phi^{\ast}e_{U} = p_{U}$. Define $E=\Phi S^{1/2}\Phi^{\ast}$. Then, for any $U\in\mathcal U$ we have 
\begin{equation}\label{ug5}
E e_{U}  = \Phi S^{1/2}\Phi^* e_{U} = \Phi S^{1/2}p_{U}=\Phi S^{1/2}S^{-1/2}Uf=\Phi Uf,
\end{equation}
which proves (iii). Hence, for $U,V \in \mathcal U$,
\[
\langle E^2 e_U, e_V \rangle = \langle E e_U, E e_V \rangle = \langle \Phi Uf, \Phi Vf \rangle = \langle Uf, Vf \rangle.
\]
This shows that $E^2$ is the Gramian of $\mathcal Uf$.

Moreover, using \eqref{ug5} we have
\[\norm{Eg}^{2} = \sum_{U\in\mathcal{U}}|\langle Eg,e_{U}\rangle|^{2} = \sum_{U\in\mathcal{U}}|\langle g,Ee_{U}\rangle|^{2} = \sum_{U\in\mathcal{U}}|\langle g,\Phi Uf \rangle|^{2}.\]
Since $\Phi$ is unitary, $\Phi(\mathcal{U}f)$ is a frame for $\Phi(\Hil)$ with optimal frame bounds $A^{2}$ and $B^{2}$. The frame property now implies 
\[
A^{2}\norm{g}^{2}\leq\norm{Eg}^{2}\leq B^{2}\norm{g}^{2}
\qquad\text{for all }g\in \Phi(\Hil),
\]
which shows (ii).

Finally, to show the last part of Theorem \ref{ug2}, we define $U_{0}:\Hil\oplus\Phi(\Hil)^{\bot}\to\ell^{2}(\mathcal U)$ by
\[U_{0}g=\left\{\begin{array}{ll} \Phi g & g\in\Hil,
\\ g & g\in\Phi(\Hil)^{\bot}.\end{array}\right.\]
It is clear that $U_{0}$ is unitary, since $\Phi:\Hil\to\Phi(\Hil)$ is an isometric isomorphism. Since $\Phi^*\Phi$ is the identity on $\Hil$,
\[
E^{2}=\Phi S^{1/2}\Phi^* \Phi S^{1/2}\Phi^* =\Phi S\Phi^*.
\]
Hence,
\[E^{2}U_{0}g =
\begin{cases} E^{2}\Phi g = \Phi Sg = U_{0}Sg, & g\in\Hil
, \\
 E^{2}g = \Phi S\Phi^*g =  0 = U_{0}{\bf{0}}g &g\in\Phi(\Hil)^{\bot}.
 \end{cases}
\]
Thus, $E^2=U_0 (S \oplus \mathbf{0}) U_0^*$, which completes the proof of Theorem \ref{ug2}.
\end{proof}

\section{Frame vectors for projective unitary representations}

In this section we establish a variant of Theorem \ref{ug} in the context of projective unitary representations \cite{HL08}. Initially, Han and Larson formulated their dilation theorem in terms of unitary group representations in \cite[Theorem 3.8']{hl}. Since this setting does not include Gabor systems, Gabardo and Han \cite{GH} established the dilation theorem for group-like unitary systems, which covers Gabor systems as a special case. Recently, Han and Larson \cite{HL08}  adapted these arguments to the even more general setting of projective unitary representations \cite{Va}.

Let $G$ be a countable group and $\mathcal U(\mathcal H)$ be the group of unitary operators on a separable Hilbert space $\mathcal H$. A {\it projective unitary representation} is a mapping $\pi: G \to \mathcal U(\mathcal H)$ such that
\[
\pi(g)\pi(h) = \mu(g,h) \pi(gh) \qquad\text{for all }g,h \in G,
\]
where $\mu: G \times G \to \T$ is a {\it multiplier} of $\pi$. To emphasize the dependence of $\mu$, we also say that $\pi$ is a $\mu$-projective unitary representation. Any multiplier $\mu$ for $G$ must satisfy
\begin{align}
\label{mu1}
&\mu(g_1,g_2g_3)\mu(g_2,g_3)=\mu(g_1g_2, g_3) \mu(g_1,g_2), \qquad g_1,g_2,g_3 \in G,
\\
\label{mu2}
&\mu(g,\mathbf e) = \mu (\mathbf e,g) =1, \qquad g\in G, \text{ where $\mathbf e$ is the group unit of $G$}.
\end{align}

Similar to the case of unitary representations, a prominent role is played by the left (and right) regular projective representations. Let $\mu$ be a multiplier for $G$. For $g\in G$, we define a unitary operator $\lambda_g : \ell^2(G) \to \ell^2(G)$ by
\[
\lambda_g e_h = \mu(g,h) e_{gh}, \qquad h\in G,
\]
where $\{e_g: g\in G\}$ is the standard orthonormal basis of $\ell^2(G)$. By \eqref{mu1} and \eqref{mu2}, $\lambda$ is  a $\mu$-projective unitary representation of $G$, which is called the {\it left-regular $\mu$-projective representation}. Likewise, we could define the right-regular $\mu$-projective representation, but we shall only need its left variant in the proof of Theorem \ref{rep}.

\begin{thm}\label{rep}
Suppose $G$ is a group, $\pi$ is a projective unitary representation of $G$ on a Hilbert space $\Hil$ with a multiplier $\mu$, and $\{\pi(g)f\}_{g\in G}$ is a frame for $\mathcal H$. Then,
\begin{enumerate}
\item\label{repi} There exists a vector $f_{1}\in\Hil$ such that $\{\pi(g)f_{1}\}_{g\in G}$ is a Parseval frame for $\Hil$. If $S$ is the frame operator of $\{\pi(g)f\}_{g\in G}$, then 
\begin{equation}\label{rep1}
S^{1/2}(\pi(g)f_{1}) = \pi(g)f \qquad\text{for all }g\in G.
\end{equation}

\item\label{repii} There exist a $\mu$-projective unitary representation $\mathcal \pi'$ of $G$ on a Hilbert space $\mathcal K$ and a vector $f_{2}\in\mathcal K$ such that $\{\pi'(g)f_{2}\}_{g\in G}$ is a Parseval frame for $\Kil$ and $\{\pi(g)f_{1}\oplus\pi'(g)f_{2}\}_{g\in G}$ is an orthonormal basis for $\Hil\oplus\Kil$.
\end{enumerate}
\end{thm}

\begin{proof} First, we will show that $\pi(g)$ and $S$ commute for all $g\in G$, and thus $\pi(g)$ and $S^{-1/2}$ commute. Let $g\in G$ and $x\in\Hil$
\begin{align*}
\pi(g)Sx & = \pi(g)\left(\sum_{h\in G}\langle x,\pi(h)f\rangle\pi(h)f\right) = \sum_{h\in G}\langle \pi(g)x,\pi(g)\pi(h)f\rangle \pi(g)\pi(h)f\\
 &=\sum_{h\in G}\langle \pi(g)x,\mu(g,h)\pi(gh)f\rangle \mu(g,h)\pi(gh)f = \sum_{h'\in G}\langle \pi(g)x,\pi(h')f\rangle \pi(h')f = S\pi(g)x.
\end{align*}
By Proposition \ref{di1} the system $\{S^{-1/2}\pi(g)f\}_{g\in G}$ is a Parseval frame. Since $S^{-1/2}\pi(g)f=\pi(g)S^{-1/2}f$, by defining $f_{1} = S^{-1/2}f$ we have shown part \eqref{repi}. 

Next we will show \eqref{repii}. Let $\{e_{g}\}_{g\in {G}}$ be the coordinate basis for $\ell^{2}({G})$. Let $\Phi:\Hil\to\ell^{2}( G)$ be  the analysis operator of the Parseval frame $\{\pi(g)f_1\}_{g\in G}$. Consequently, $\Phi$ is an isometry. Let $P$ be the orthogonal projection onto $\Phi(\Hil)$. Let $\lambda$ be the left-regular $\mu$-projective representation of $G$ as defined above.
Note that $\pi(g^{-1}) = \mu(g,g^{-1})\pi^*(g)$. Hence, for $x\in\Hil$ we have
\begin{equation}\label{rep2}\begin{split}
\Phi \pi(g) x & = \sum_{h\in G}\langle \pi(g) x,\pi(h)f_{1}\rangle e_{h} = \sum_{h\in G}\langle x,\overline{\mu(g,g^{-1})}\mu(g^{-1},h)\pi(g^{-1}h)f_{1}\rangle e_{h}
\\
 & = \lambda_{g} \bigg(\sum_{h\in G}\langle x,\pi(g^{-1}h)f_{1}\rangle 
 \mu(g,g^{-1}) \overline{\mu(g,g^{-1}h)\mu(g^{-1},h)} e_{g^{-1}h} \bigg) = \lambda_{g}\Phi x
\end{split}\end{equation}
In the penultimate step we used the identity $\lambda_g e_{g^{-1}h} = \mu(g,g^{-1}h) e_h$ and in the last step we used \eqref{mu1} and \eqref{mu2} to eliminate three multiplier terms. Hence, we have established the identity
\begin{equation}\label{rep2a}
\pi(g) = \Phi^* \lambda_g \Phi \qquad\text{for all }g\in G.
\end{equation}

Since $\{\pi(g)f_{1}\}_{g\in G}$ is a Parseval frame for $\Hil$ and $\Phi$ is an isometry, $\{\Phi\pi(g)f_{1}\}_{g\in G}$ is a Parseval frame for $\Phi(\Hil)$. 
%From \eqref{rep2} we see that this is the same frame as $\{\lambda_{g}\Phi(f_{1})\}_{g\in G}$. 
Thus, we have
\begin{equation}\label{rep3}\begin{split}
Pe_{g} & = \sum_{h\in G}\langle Pe_{g},\Phi\pi(h)f_{1}\rangle \Phi\pi(h)f_{1} = \sum_{h\in G}\langle e_{g},P\Phi\pi(h)f_{1}\rangle \Phi\pi(h)f_{1}\\
 & = \Phi \left(\sum_{h\in G}\langle e_{g},\Phi\pi(h)f_{1}\rangle \pi(h)f_{1}\right) = \Phi \left(\sum_{h\in G}\langle \pi(g)f_{1},\pi(h) f_{1}\rangle \pi(h)f_{1}\right) = \Phi(\pi(g)f_{1}).
\end{split}\end{equation}
Using \eqref{rep2} and \eqref{rep3} for any $h\in G$ we have
\[
\lambda_{g}Pe_{h} = \lambda_{g}\Phi(\pi(h)f_{1}) = \Phi(\pi(g)\pi(h)f_{1}) = \mu(g,h) \Phi(\pi(gh)f_{1}) =  \mu(g,h)  Pe_{gh} = P\lambda_{g}e_{h}.
\]
Since $\lambda_{g}P$ and $P\lambda_{g}$ agree on elements of the coordinate basis of $\ell^2(G)$, we have established the commutation relation
\begin{equation}\label{rep3a}
\lambda_g P = P \lambda_g \qquad\text{for all }g\in G.
\end{equation}

For $g\in G$ define $\pi'(g) = (I-P)\lambda_{g}$. Let $\mathbf e \in G$ be the group unit, and define $f_{2} = (I-P)e_{\mathbf e}$.
Using \eqref{rep3a} we have
\[\pi'(g)f_{2} = (I-P)\lambda_{g}f_{2} = (I-P)\lambda_{g}(I-P)e_{\mathbf e} = (I-P)^{2}\lambda_{g}e_{\mathbf e} = (I-P)e_{g}.\]
Since $(I-P)$ is an orthogonal projection, this shows that $\{\pi'(g)f_{2}\}_{g\in G}$ is a Parseval frame for $\Kil=\Phi(\Hil)^{\bot}$. For $g,h\in G$ we have
\begin{align*}
\pi'(gh) & = (I-P)\lambda_{gh} = \mu(g,h) (I-P)\lambda_{g}\lambda_{h} = \mu(g,h) (I-P)^{2}\lambda_{g}\lambda_{h}\\
 & = \mu(g,h)(I-P)\lambda_{g}(I-P)\lambda_{h} = \mu(g,h) \pi'(g)\pi'(h).
\end{align*}
Moreover, for any $g\in G$,
\[
(\pi'(g))^{\ast} = \lambda_{g}^{\ast} (I-P)^{\ast} =\overline{\mu(g,g^{-1})} \lambda_{g^{-1}} (I-P)= \overline{\mu(g,g^{-1})} (I-P)\lambda_{g^{-1}} = \overline{\mu(g,g^{-1})}\pi'(g^{-1}),
\]
which implies that $(\pi'(g))^*\pi'(g) = \pi'(g)(\pi'(g))^* = \pi'(\mathbf e)=I-P$. Thus, by restricting the domain of $\pi'(g)$ to $\Kil$, $\pi'(g)$ becomes a unitary operator on $\Kil$. This shows that $\pi'$ is a $\mu$-projective unitary representation of $G$ on $\Kil$.

Define the map $\Psi:\ell^{2}( G)\to\Hil\oplus\Kil$ by 
\[
\Psi f = \begin{cases}
\Phi^{-1}f & \text{for }f\in \Phi(\Hil), \\
f& \text{for } f\in\Kil.
\end{cases}
\]
Clearly, $\Psi$ is an isometric isomorphism, since $\Phi$ is an isometric isomorphism between $\Hil$ and $\Phi(\Hil)$. Note that 
\[
e_{g} = Pe_g + (I-P)e_g=\Phi(\pi(g)f_{1}) + \pi'(g)f_{2}.
\]
Since $\{e_{g}\}_{g\in G}$ is an orthonormal basis of $\ell^2( G)$, so is $\{\Psi e_{g}\}_{g\in G}$, where
\[
\Psi e_{g} = \pi(g)f_{1}\oplus\pi'(g)f_{2}.
\]
This completes the proof of Theorem \ref{rep}.
\end{proof}

We shall illustrate Theorem \ref{rep} by showing that any Gabor type unitary system has a dilation property. For $(t,s) \in \R^n \times \R^n$, define the time-frequency (Gabor representation) $\pi_0$ on $L^2(\R^n)$ by
\[
\pi(t,s) = M_s T_t,
\]
where $M_s f(x) = e^{i \langle x, s \rangle }f(x)$ and $T_t f(x) =f (x-t)$ for $f\in L^2(\R^n)$. Let $G$ be a full rank lattice in $\R^n \times \R^n$, that is $G= P(\Z^n \times \Z^n)$ for some $(2n) \times (2n)$ invertible real matrix. Observe that
\[
\pi_0(t,s)\pi_0(t',s') = e^{-i \langle t,s' \rangle} \pi_0(t+t', s+s').
\]
Thus, the Gabor representation $\pi |_G$ is a projective unitary representation with the multiplier
\begin{equation}\label{mult}
\mu((t,s),(t',s')) =  e^{-i \langle t,s' \rangle}
\qquad\text{for }(t,s), (t',s') \in G \subset \R^n \times \R^n.
\end{equation}
For a general Hilbert space $\Hil$ we adopt the following notion.

\begin{defn}
Let $\mathcal H$ be a separable Hilbert space and $G$ be a a full rank lattice in $\R^n \times \R^n$.  We say that $\pi$ is a {\it Gabor type representation} of the lattice $G$, if $\pi: G \to \mathcal U (\mathcal H)$ is a $\mu$-projective unitary representation with the multiplier $\mu$ given by \eqref{mult}.
\end{defn}

As an immediate application of Theorem \ref{rep} we have

\begin{thm}\label{rep-gab}
Suppose $G$ is a full rank lattice and $\{\pi(g)f\}_{g\in G}$ is a Gabor frame for $\mathcal H$ with optimal bounds $A^{2}$ and $B^{2}$. Then, there exist:
\begin{enumerate}
\item\label{rep-gi} a Gabor type representation $\mathcal \pi'$ of the lattice $G$ on a Hilbert space $\mathcal K$,
\item\label{rep-gii} a complete wandering vector $ e=f_1 \oplus f_2 \in\Hil \oplus \mathcal K$ for $\mathcal \pi \oplus \mathcal \pi'$, and
\item\label{rep-giii} a positive operator $E:\Hil \to\Hil$ with $\{A,B\}\subseteq\sigma(E)\subseteq [A,B]$, 
such that
\begin{equation}\label{rep-g1}
E(\pi(g)f_{1}) = \pi(g)f \qquad\text{for all }g\in G.
\end{equation}
\end{enumerate}
\end{thm}

In the case when $G = a\Z \times b \Z$, $a,b>0$, a Gabor type representation is uniquely determined by two generators $\{M,T\}$ satisfying $TM= e^{-iab}MT$. Indeed, $\pi(an,bm)= M^m T^n$, where $M=\pi(0,b)$ and $T=\pi(a,0)$. In this case, we can simply say that $\{M^mT^n: m,n\in\Z\}$ is a Gabor type unitary system of $G = a\Z \times b \Z$.
Then, Theorem \ref{rep-gab} yields Corollary \ref{gab} which generalizes the result of Han and Larson \cite[Theorem 4.8]{hl} from the setting of Parseval frames to general non-tight frames.

\begin{cor}\label{gab}
Let $\mathcal{U}=\{M^mT^n: m,n\in\Z\}$ be a Gabor type unitary system on a Hilbert space $\Hil$ of the lattice $a\Z \times b \Z$. Suppose that $f \in\Hil$ is a frame vector for $\mathcal U$ with optimal bounds $A^{2}$ and $B^{2}$.
Then, there exist:
\begin{enumerate}
\item 
a Gabor type unitary system $\mathcal{U}'=\{(M')^m(T')^n: m,n\in\Z\}$ on a Hilbert space $\mathcal{K}$ of the lattice $a\Z \times b \Z$,
\item
a complete wandering vector $e=f_1 \oplus f_2 \in\Hil \oplus \mathcal K$ for 
\[
\mathcal U \oplus \mathcal U':=\{(M^mT^n) \oplus ((M')^m(T')^n): m,n\in\Z\}, {\it{ and}}
\]
\item
a positive operator $E:\Hil \to\Hil$ with $\{A,B\}\subseteq\sigma(E)\subseteq [A,B]$, 
such that
\[
E(M^mT^nf_1) = M^mT^nf \qquad\text{for all }m,n \in \Z.
\]
\end{enumerate}
\end{cor}

\section{Dilations of frame wavelets}

In this section we generalize the result of Dutkay, Han, Picioroaga, and Sun \cite{dhps} from the setting of Parseval wavelets to non-tight frame wavelets. We find optimal conditions on non-tight frame wavelets for which a generalization of \cite[Theorem 2.6]{dhps} holds. It turns out that our dilation result is not possible for arbitrary non-tight frame wavelets. This marks a significant distinction between the previously considered situations of Gabor systems and unitary groups and the case of wavelets. In particular, one can not expect that every dilation result will extend automatically from the tight to the non-tight setting.

Following \cite{dhps} we adopt the following definition.
The Baumslag-Solitar group \cite{bs} is given by the group presentation
\[
BS(1,2):=\langle d,t | dtd^{-1} = t^2 \rangle.
\]
Thus, any unitary representation of $BS(1,2)$ is given by two unitary operators $D$ and $T$ on some Hilbert space $\Hil$, that satisfy the relation $DTD^{-1} = T^2$.

\begin{defn}
Let $\{D,T\}$ be a representation of the Baumslag-Solitar group $BS(1,2)$ on a Hilbert space $\Hil$. We say that $\psi \in \Hil$ is a frame (or orthonormal) wavelet for $\{D,T\}$ if $\{D^jT^k \psi: j,k\in \Z\}$ is a  frame (or orthonormal basis) for $\Hil$, respectively.
\end{defn}

In order to state our dilation theorem for frame wavelets we need one more notion. Following \cite{hl}, we define the local commutant of a set of unitary operators  $\mathcal U$ of $\Hil$ at $\psi\in \Hil$ as
\[
C_\psi(\mathcal U) = \{T \in B (\mathcal H): TU \psi=UT \psi \quad\forall U \in\mathcal U\}.
\]

\begin{thm}\label{wave}
Let $\{D,T\}$ be a representation of the Baumslag-Solitar group $BS(1,2)$ on a Hilbert space $\Hil$, and let $\psi \in\Hil$ be a frame wavelet for $\{D,T\}$ with optimal bounds $A^{2}$ and $B^{2}$. Let $S$ be the frame operator of $\{D^jT^k \psi: j,k\in \Z\}$, and
assume that $S^{-1/2}$ is in the local commutant
\begin{equation}\label{s-}
S^{-1/2} \in \mathcal C_\psi(\{D^jT^k: j,k\in\Z\}).
\end{equation}
Then, there exist:
\begin{enumerate}
\item 
a representation $\{D',T'\}$ of $BS(1,2)$ on a Hilbert space $\mathcal K$,
\item
an orthonormal wavelet $\psi_1 \oplus \psi_2 \in\Hil \oplus \mathcal K$ for 
the representation $\{D\oplus D',T \oplus T'\}$ of $BS(1,2)$ on $\Hil \oplus \mathcal K$, and
\item
a positive operator $E:\Hil \to\Hil$ with $\{A,B\}\subseteq\sigma(E)\subseteq [A,B]$, 
such that
\begin{equation}\label{se}
E(D^jT^k \psi_1) = D^jT^k\psi \qquad\text{for all }j,k \in \Z.
\end{equation}
\end{enumerate}
Conversely, if (i)--(iii) hold, then \eqref{s-} necessarily holds.
\end{thm}

\begin{proof}
Define $\psi_1 = S^{-1/2} \psi$. By our assumption \eqref{s-},
\begin{equation}\label{s-1}
S^{-1/2} (D^j T^k \psi) = D^j T^k \psi_1 \qquad\text{for all }j,k\in \Z.
\end{equation}
Since, $\{S^{-1/2} (D^j T^k \psi): j,k\in \Z\}$ is a Parseval frame in $\mathcal H$, $\psi_1$ is a Parseval wavelet for $\{D,T\}$. Applying \cite[Theorem 2.6]{dhps} yields a representation $\{D',T'\}$ of $BS(1,2)$ on a Hilbert space $\mathcal K$ satisfying (ii). Moreover, (iii) holds for $E:=S^{1/2}$ by \eqref{s-1}.

Conversely, assume (i)--(iii). By \eqref{se} the frame operator $S$ satisfies for $f\in\mathcal H$,
\[
\begin{aligned}
Sf = \sum_{j,k\in \Z} \langle f, D^jT^k\psi \rangle D^j T^k \psi &
= \sum_{j,k\in \Z} \langle f, E D^jT^k\psi_1 \rangle E D^j T^k \psi_1
\\
& = E\bigg(  \sum_{j,k\in \Z} \langle Ef, D^jT^k\psi_1 \rangle D^j T^k \psi_1 \bigg) =E^2f.
\end{aligned}
\]
In the last step we used the fact that $\{D^j T^k \psi_1: j,k\in \Z\}$ is a Parseval frame in $\mathcal H$, which is a consequence of (ii). Hence, $E=S^{1/2}$. Letting $j=k=0$ in \eqref{se} yields $S^{1/2}\psi_1= \psi$. Thus, \eqref{se} takes the form of \eqref{s-1}. By definition, \eqref{s-1} is equivalent to the local commutant property \eqref{s-}. This completes the proof of Theorem \ref{wave}.
\end{proof}

\begin{remark}
The local commutant assumption \eqref{s-} is a convenient way of stating that the canonical Parseval frame of the affine system $\{D^j T^k \psi: j,k\in \Z\}$ is also an affine system $\{D^j T^k \psi_1: j,k\in \Z\}$ for some $\psi_1 \in\mathcal H$.
If $\psi$ is a Parseval wavelet for $\{D,T\}$, then \eqref{s-} is automatically satisfied since $S=Id$. Thus, Theorem \ref{wave} provides a generalization of the result of Dutkay, Han, Picioroaga, and Sun \cite[Theorem 2.6]{dhps}. At the same time, Theorem \ref{wave} asserts that dilation results  of this kind are only possible for frame wavelets whose canonical Parseval frame is an affine system. This is in a stark contrast with Gabor systems, where Corollary \ref{gab} holds regardless of such a priori assumption. In retrospect, this is not that surprising since the canonical Parseval frame of a Gabor system is always known to be a Gabor system.
\end{remark}

\bibliographystyle{amsplain}

\end{document}